\documentclass[10pt,a4paper]{amsart}%
\usepackage{charter}
\usepackage[english]{babel}
\usepackage{amsmath}
\usepackage{amsfonts}
\usepackage{amssymb}
\usepackage{graphicx}
\usepackage{t1enc}
\usepackage{mathrsfs}
\usepackage{amsthm}
\usepackage[colorinlistoftodos]{todonotes}%

\providecommand{\U}[1]{\protect\rule{.1in}{.1in}}
\newtheorem{theorem}{Theorem}[section]

\newtheorem{corollary}[theorem]{Corollary}

\newtheorem{lemma}[theorem]{Lemma}

\newtheorem{proposition}[theorem]{Proposition}

\theoremstyle{definition}
\newtheorem{definition}[theorem]{Definition}

\newtheorem{remark}[theorem]{Remark}

\DeclareMathOperator{\EqFK}{Eq^{FK}}
\DeclareMathOperator{\EqFKe}{Eq^{FK}_\varepsilon}
\DeclareMathOperator{\EqFKn}{Eq^{FK}_{1/n}}
\newcommand{\fbar}{\bar{f}}

\newcommand{\eps}{\varepsilon}
\DeclareMathOperator{\diam}{diam}

\newcommand{\orb}{orb}

\renewcommand{\phi}{\varphi}

\newcommand{\Z}{{\mathbb{Z}}}

\newcommand{\N}{\mathbb{N}}

\newcommand{\Dom}{\mathcal{D}}
\newcommand{\Ran}{\mathcal{R}}

\newcommand{\dbar}{\bar{d}}

\newcommand{\bP}{\mathbf{P}}

\title{On topological models of zero entropy loosely Bernoulli systems}
\author{Felipe Garc\'{i}a-Ramos}
\address{CONACyT \& Instituto de Fisica,
Universidad Autónoma de San Luis Potosí (UASLP)
Av. Manuel Nava \#6, Zona Universitaria
San Luis Potosí, S.L.P., 78290, México}
\email{felipegra@yahoo.com}
\author{Dominik Kwietniak}
\address{Faculty of Mathematics and Computer Science,
  Jagiellonian University in Krakow, ul. {\L}ojasiewicza 6, 30-348 Krak\'ow, Poland.}
\email{dominik.kwietniak@uj.edu.pl}

\date{\today}
\begin{document}
\begin{abstract}
We provide a purely topological characterisation of uniquely ergodic topological dynamical systems (TDSs) whose unique invariant measure is zero entropy loosely Bernoulli (following Ratner,  we call such measures loosely Kronecker). At the heart of our proofs lies Feldman-Katok continuity (FK-continuity for short), that is, continuity with respect to the change of metric to the Feldman-Katok pseudometric. Feldman-Katok pseudometric is a  topological analog of f-bar (edit) metric for symbolic systems.
We also study an opposite of FK-continuity, coined FK-sensitivity.
We obtain a version of Auslander-Yorke dichotomies: minimal TDSs are either FK-continuous or FK-sensitive, and transitive TDSs are either almost FK-continuous or FK-sensitive.
\end{abstract}
\maketitle
\section{Introduction}
The fundamental problem in dynamics is to classify dynamical systems up to conjugacy. By dynamical system we mean a space $X$ with some structure and a map $T\colon X\to X$ that preserves the structure. Given two dynamical systems, $T\colon X\to X$ and $S\colon Y\to Y$, a conjugacy is given by invertible structure-preserving map $\phi\colon X\to Y$ that intertwines with the dynamics, that is, satisfies $S\circ\phi=\phi\circ T$.

We are interested in two classes of dynamical systems: measure-preserving systems studied in \emph{ergodic theory}, and topological dynamical systems, being the subject of \emph{topological dynamics}.
A \emph{measure-preserving system} (MPS for short) is a quadruple $(X,\Sigma_X,\mu,T)$, where
$\Sigma_X$ is a $\sigma$-algebra  and $\mu$ is a probability measure preserved by $T$ such that the probability space $(X,\Sigma_X,\mu)$ is a Lebesgue space. A \emph{topological dynamical system} (TDS for short) is a pair $(X,T)$, where $X$ is a compact metrisable space, and $T$ is continuous. The conjugacy in ergodic theory and topological dynamics is known as \emph{isomorphism} and \emph{(topological) conjugacy}, respectively.

There are many connections and parallels between MPSs and TDSs (for a survey of these relations see \cite{glasner2006interplay,huangsurvey}). In particular, a theorem due to Krylov and Bogoliubov says that for every TDS $(X, T)$, we can find a Borel $T$-invariant probability measure $\mu$ on $X$, and obtain a MPS  $(X,\bar{\mathcal{B}}_X,\mu,T)$ where $\mathcal{B}_X$ is the Borel $\sigma$-algebra on $X$, and $\bar{\mathcal{B}}_X$ is its completion  with respect to $\mu$. Conversely, the Jewett-Krieger theorem says that for every ergodic MPS $(X,\Sigma_X,\mu,T)$, 
we can find a TDS $(Y,S)$ such that $S$ has a unique invariant Borel probability measure $\nu$ and the MPSs $(X,\Sigma_X,\mu,T)$ and $(Y,\bar{\mathcal{B}}_Y,\nu,S)$ are isomorphic. 
In this situation, we say $(Y,S)$ is a \emph{topological model} for $(X,\Sigma_X,\mu,T)$. We summarise our discussion by saying that \emph{every ergodic MPS has a topological model}. Every topological model restricted to the support of its unique invariant measure is minimal, hence topologically transitive, furthermore, if the modelled MPS is mixing, then every TDS model must be topologically mixing; nonetheless, there exist uniquely ergodic topologically mixing TDSs where the invariant measure is not mixing. Are there any correspondences between isomorphism-invariant properties of ergodic MPSs and topological conjugacy-invariant properties of their topological models? Is it possible to give an intrinsic characterisation of topological models for certain classes of MPSs? That is, we are asking if we can describe all topological models of an isomorphism-invariant class of MPSs using a topological conjugacy-invariant property. Our main result characterises all topological models of loosely Bernoulli systems with zero entropy.  To our best knowledge, it is the first result of this kind. 
We will discuss similar works in Section \ref{sec:related}.

Before we present the results of this paper in more detail, let us first provide some background.
In 1932, von Neumann \cite{neumann1932operatorenmethode} proved that two ergodic measure-preserving systems 
with discrete spectrum (we call these systems \emph{Kronecker}) are isomorphic if and only if their induced (Koopman) operators are spectrally isomorphic. Later on, Halmos and von Neumann \cite{halmos1942operator} proved that every Kronecker MPS admits a minimal equicontinuous topological model $(X, T)$ (i.e. $\{T^n: n\ge 0\}$ is an equicontinuous family of maps from $X$ to $X$). 
The other remarkable achievement on the isomorphism problem is Ornstein's theory, whose conclusion says that any two Bernoulli processes of equal entropy are
isomorphic \cite{ornstein70a,ornstein70b,ornstein70c} 
At the heart of Ornstein's theory lies the definition of a finitely determined process, based on the normalised Hamming distance, $\dbar_n$, between finite sequences $x_0x_1\ldots x_{n-1}$ and $y_0y_1\ldots y_{n-1}$, where 
\[
\dbar_n(x_0x_1\ldots x_{n-1},y_0y_1\ldots y_{n-1})=\frac{1}{n}\left|\{0\le j< n: x_j\neq y_j\}\right|.
\]
With the above two exceptions, the isomorphism problem seems to be hopeless \cite{foreman2004anti,foreman2011conjugacy}.

A notion of isomorphism is \emph{Kakutani equivalence} defined in \cite{kakutani43}. There is a theory, similar to Ornstein's, classifying MPSs Kakutani equivalent to either, Bernoulli or Kronecker systems (the theory is mostly due to Ornstein, Rudolph and Weiss \cite{ornsteinrudolphweiss82}, but Katok \cite{katok77} had independently obtained the same results for Kronecker systems).
Surprisingly, in this theory Kronecker and Bernoulli MPSs belong to the same family of \emph{loosely Bernoulli systems} introduced by Feldman \cite{feldman76}. Feldman followed Ornstein's approach, but replaced $\dbar_n$  in the definition of finitely determined process by the \emph{edit distance} $\fbar_n$. In this way, Feldman defined  \emph{finitely fixed process}, equivalently \emph{loosely Bernoulli systems}. Every Bernoulli process is loosely Bernoulli, but quite surprisingly there are also loosely Bernoulli processes of zero entropy. The $\fbar_n$ distance  between  $x_0x_1\ldots x_{n-1}$ and $y_0y_1\ldots y_{n-1}$ is equal to $1$ minus the fraction of letters one has to delete from each word to obtain identical (possibly empty) words; that is,
\[\fbar_n(x_0x_1\ldots x_{n-1},y_0y_1\ldots y_{n-1})=1-k/n,\]
where $k$ is the largest integer such that for some $0\le i(1)<i(2)<\ldots<i(k)<n$ and $0\le j(1)<j(2)<\ldots<j(k)<n$ it holds $x_{i(s)}=y_{j(s)}$ for $s=1,\ldots,k$.
By Abramov's formula, Kakutani equivalence preserves Kolmogorov-Sinai entropy only up to a multiplication by a positive constant. Therefore, there are at least three Kakutani equivalence classes of loosely Bernoulli processes as processes with zero, positive but finite, and infinite entropy must belong to different classes. It turns out that there are exactly three classes \cite{ornsteinrudolphweiss82} and a MPS is Kakutani equivalent to a Kronecker system if and only if it is loosely Bernoulli system and has zero entropy (\cite{ornsteinrudolphweiss82} and \cite{katok77})). Following a suggestion of Ratner (reported in \cite{feldman1979reparametrization}), we refer to zero entropy loosely Bernoulli systems as \emph{loosely Kronecker} (note that Katok \cite{katok77}) called these systems \emph{standard}).

Recently, the first author characterised ergodic measures of TDS that are isomorphic to Kronecker MPSs using a weak form of equicontinuity called \emph{$\mu$-mean equicontinuity} \cite{weakforms} (also see \cite{huang2018bounded}). The latter property is a weakening of mean equicontinuity introduced by Fomin \cite{fomin}. It has been studied in \cite{auslander1959,oxtoby,lituye,weakforms,garcia2017dynamical,downarowiczglasner,huang2018bounded,garciajagerye} (also, see the survey \cite{lisurveymean}). Mean equicontinuous systems are defined using the \emph{Besicovitch pseudometric} (see Section 2.1), a topological version of the $\dbar$ pseudometric which appears naturally in dynamics (see papers cited above and \cite{downarowicz2014,kwietniak2017}). Mean equicontinuous TDSs are a proper sub-class of topological models for Kronecker systems which can be characterised naturally \cite{lituye,downarowiczglasner}.

Since Besicovitch pseudometric, related to $\dbar$, proved to be a useful tool for characterising Kronecker systems, it is natural to expect that in order to study loosely Kronecker systems one should find a pseudometric which is a topological version of the edit distance, $\fbar$. The relation of a new pseudometric to $\fbar$ should resemble how the Besicovitch distance is a continuous analogue of the $\dbar$-distance. Such a pseudometric was introduced by {\L}{\k{a}}cka and the second author in \cite{kwietniaklacka17} who coined it the \emph{Feldman-Katok pseudometric}. Here we denote it by $\rho_{FK}$ and use it to characterise loosely Kronecker MPSs and their topological models.
\begin{theorem}[Theorems \ref{LKcharac} and \ref{uniquelyerg} combined]\label{thm:main_comb}
    Let $(X,T)$ be a TDS and $\mu$ be  an ergodic $T$-invariant Borel probability measure. Let $\mathcal{B}_X$ be the Borel $\sigma$-algebra on $X$ and let $\bar{\mathcal{B}}_X$ stand for the completion of ${\mathcal{B}}_X$ with respect to $\mu$.
\begin{enumerate}
  \item The MPS $(X,\bar{\mathcal{B}}_X,\mu,T)$ is loosely Kronecker (zero entropy loosely Bernoulli) if and only if there exists a Borel set $M\subset X$ with $\mu (M)=1$ such that $\rho_{FK}(x,y)=0$ for every $x,y\in M$.\label{first:main}
  \item $(X,T)$ is a (uniquely ergodic) topological model of a loosely Kronecker MPS if and only if $\rho_{FK}(x,y)=0$ for every $x,y\in X$.\label{second:main}
\end{enumerate}
\end{theorem}

We call systems satisfying the second part of Theorem \ref{thm:main_comb} \emph{topologically loosely Kronecker}.

The first part of Theorem \ref{thm:main_comb} is an analogue of Corollary 39 from \cite{weakforms} (see also Theorem \ref{chakro}) for Kakutani equivalence, but our proof uses different techniques. The second part of Theorem \ref{thm:main_comb} (see also Theorem \ref{uniquelyerg} and Corollary \ref{cor}) is much more surprising. It does not have any direct analogue for the Besicovitch pseudometric as no purely topological characterisation of models of Kronecker MPSs is known. The notion of $\mu$-mean equicontinuity, used to characterise Kronecker MPSs in \cite{weakforms}, is a topological/measure-theoretic hybrid (for more information see the discussion after Corollary \ref{cor}).

In the literature, there are several examples of uniquely ergodic TDSs with the unique invariant measure being loosely Kronecker systems  (see Section 6.1). Our results indicate that these systems must be topologically loosely Kronecker, thus providing new information on their dynamics. Furthermore, the notion of Feldman-Katok continuity leads us to an opposite notion of Feldman-Katok sensitivity, and we obtain a new classification of minimal (and transitive) TDSs. Recall first that the opposite (but not mutually exclusive in general) of equicontinuity is \emph{sensitivity to initial conditions} (or simply sensitivity). A TDS $(X,T)$ is \emph{sensitive} if there exists $\varepsilon>0$ such that for every nonempty open set $U\subset X$, there exist $n>0$ and $x,y\in U$ such that $d(T^nx,T^ny)>\varepsilon$. Auslander and Yorke \cite{auslander1980} showed that for minimal TDSs these notions do complement one other, that is, a minimal TDS is equicontinuous if and only if it is not sensitive. A somewhat relaxed version of Auslander-Yorke dichotomy holds for transitive systems \cite{akin1996transitive}.
Similarly, using the Besicovitch pseudometric one can define \emph{mean equicontinuity} and \emph{mean sensitivity} and show that these notions are also mutually exclusive for minimal systems and almost mutually exclusive for transitive TDSs. We replace Besicovitch pseudometric with the Feldman-Katok pseudometric in the definitions of mean equicontinuity/sensitivity and we introduce \emph{Feldman-Katok continuity} and a new notion of chaotic sensitivity, coined \emph{FK-sensitivity}. We also prove new Auslander-Yorke type dichotomies for minimal and transitive systems (Theorem \ref{thm:duality}).
As a consequence, we obtain new information on the dynamics of minimal TDSs with positive topological entropy: we prove they are always FK-sensitive (Corollary \ref{cor:min-pos-ent}). To our best knowledge, FK-sensitivity is the strongest form of sensitivity appearing in all minimal TDSs with positive entropy. It might be slightly surprising for the reader accustomed to Besicovitch pseudometric, that a transitive TDS is Feldman-Katok continuous if and only if it is topologically loosely Kronecker (that is, $\rho_{FK}(x,y)=0$ for all $x,y\in X$, as in the second part of Theorem \ref{thm:main_comb}).

\subsection{Related results}\label{sec:related}
In general, there is no apparent connection between properties of measure-preserving systems and their topological models. It is known that every ergodic (in particular non-mixing) MPS has a topological model which is topologically mixing \cite{lehrer1987topological}. In particular, this implies that any ergodic nilsystem/distal/rigid MPS admits a topological model that is not a topological nilsystem/distal/rigid. By the variational principle for topological entropy, the measure-theoretic entropy of a MPS equals the topological entropy of any of its topological models. However, this correspondence fails even for related notions like completely positive entropy \cite{glasner1994strictly}, sequence entropy, and other notions of low complexity \cite{kerr2007independence,fuhrmann2020,fuhrmann2018irregular}.

Also, it is known that every Kronecker MPS is measurable distal, and every measurable distal MPS is loosely Kronecker, but each property has a distinct relation to the topological dynamics of their models. There are many properties that can be called ``topological Kronecker'' for minimal TDSs: equicontinuity, nullness, tameness, regularity, and mean equicontinuity. It turns out that there exist topological models for Kronecker MPSs \emph{properly} in each of these classes (equicontinuous; null, but not equicontinuous; tame, but not null; etc., see \cite{halmos1942operator,kerr2007independence,fuhrmann2020,fuhrmann2018irregular,garciajagerye}) and there are even models lying outside these classes \cite{lehrer1987topological,garcia2017dynamical}. When it comes to measurably distal MPSs there is only one topological analog: \emph{topological distality}. Lindenstrauss \cite{lindenstrauss99} proved that there exist ergodic distal MPSs that admit no topological distal model; nonetheless, he showed that every ergodic distal MPS admits a topological distal realisation (a distal, not necessarily uniquely ergodic TDS with an invariant measure isomorphic to an original MPS). The motivation for that work came from Beleznay and Foreman's questions regarding descriptive set-theoretical properties of distal systems (see \cite{foreman95}). This line of research has many important applications, but discussing them here would drift as away from our main topic, therefore we refer the reader to \cite{gutman2019strictly} and references therein.

\subsection{Organisation of the paper}
The rest of this paper is organised as follows. Section \ref{sec:basic} contains basic preliminaries and the definition of the Feldman-Katok pseudometric. Section \ref{sec:FK-cont} contains the definition of topological and measure-theoretic Feldman-Katok continuity and sensitivity, and their basic properties. Sections \ref{sec:LK} and \ref{sec:TLK} contain the main proofs of the main results of the paper. In Section \ref{sec:examples} we give examples and applications of our results. In the appendix we present some proofs for completeness.

\section{Basic definitions and notation}\label{sec:basic}

In this section, after setting up some notation and terminology, we recall the definition of Feldman-Katok pseudometric. 

\subsection{Topological dynamics}
Let $X$ be a metric space. The closed ball of radius $\varepsilon>0$ around $x\in X$ will be denoted by $B_{\varepsilon}(x)$.
Let $\mu$ be a Borel measure on $X$. The collection of Borel subsets of $X$ (respectively, Borel subsets of $X$ with positive measure)  will be denoted by $\mathcal{B}_{X}$ (respectively, by $\mathcal{B}^+_{X}$).

We say a pair, $(X,T)$, is a \textbf{topological dynamical system (TDS)} if $X$ is a compact metrisable space (with a compatible metric $d$) and $T\colon X\rightarrow X$ is a continuous map. We denote the (forward) \textbf{orbit} of $x$ by $\orb(x):=\{T^n(x):n\ge 0\}$.
A TDS $(X,T)$ is:
\begin{itemize}
  \item \textbf{minimal} if for every $x\in X$, the set $orb(x)$ is dense in $X$,
  \item \textbf{transitive} if for every pair of nonempty open sets $U,V\subset X$, there exists $n>0$ such that $T^{-n}U\cap V\neq \emptyset$,
      \item \textbf{uniquely ergodic} if there is only one $T$-invariant Borel probability measure, and
      \item \textbf{strictly ergodic} if it is minimal and uniquely ergodic.
\end{itemize}

The \textbf{support} of a Borel measure is the smallest closed set of full measure. The support of an invariant measure is \textbf{full}, if it is equal to $X$.
A uniquely ergodic TDS is minimal if and only if the $T$-invariant Borel measure has full support.

Recall that two TDSs, $(X,T)$ and $(X',T')$, are \textbf{topologically conjugate} if there exists a bijective continuous function $\phi\colon X\rightarrow X'$ such that $\phi\circ T=T'\circ\phi$.
\subsection{Ergodic theory}
We say $(X,\Sigma_X,\mu,T)$ is a \textbf{measure preserving system (MPS)} if $(X,\Sigma_X,\mu)$ is a Lebesgue probability space and $T\colon X\rightarrow X$ is a measure-preserving transformation. If $X$ is a metric space and $T$ is continuous, $\Sigma_X$ will always denote the completion of the Borel sigma-algebra with respect to $\mu$ and we will omit writing it, that is we will write simply $(X,\mu,T)$.
A $T$-invariant probability measure $\mu$ is \textbf{ergodic} if every invariant measurable set has measure one or zero. In this case, we also say the associated MPS is \textbf{ergodic}.

Two MPSs, $(X,\Sigma_X,\mu,T)$ and $(X',\Sigma_X',\mu',T')$, are \textbf{isomorphic} if there exists a bijective bi-measure-preserving function $\phi\colon X\rightarrow X'$ such that $\phi\circ T=T'\circ\phi$.
A MPS is \textbf{Kronecker} if it is isomorphic to a rotation on a compact abelian group equipped with the Haar measure.
\begin{definition}
	\label{def:kakutani}
Two MPSs $(X,\Sigma,\mu,T)$ and $(X^{\prime},\Sigma^{\prime},\mu^{\prime},T^{\prime})$ are \textbf{Kakutani equivalent} if there exist measurable
sets $A\subset X$ and $A^{\prime}\subset X^{\prime}$ with $\mu(A)\cdot\mu'(A')>0$ such that $(A,T_{A}%
,\mu_{A})$ and  $(A^{\prime},T_{A^{\prime}},\mu_{A^{\prime}})$ are isomorphic,
where $T_{A}\colon A\rightarrow A$ denotes the induced transformation (or the first return
map) and $\mu_{A}$ is the induced measure on $A$.
\end{definition}

From the definition it is clear that if two systems are isomorphic, then they must be Kakutani equivalent. The converse is false \cite{feldman76}.
\begin{definition}
	A MPS is \textbf{loosely Kronecker} if it is Kakutani equivalent to a Kronecker MPS.
\end{definition}

\begin{remark}
\begin{enumerate}
  \item A MPS is loosely Kronecker if and only if it is loosely Bernoulli and has zero entropy \cite{katok77}.
  \item  Any two Kronecker MPSs are Kakutani equivalent \cite{ornsteinrudolphweiss82}.
  \item The sets $A$ and $A'$ (in Definition \ref{def:kakutani}) can be chosen to have arbitrarily large (but not full) measure, and even one of them could have full measure \cite{katok77,ornsteinrudolphweiss82}.
  \item It follows from Abramov's formula for the entropy of an induced transformation that if two MPSs are Kakutani equivalent then either both have zero entropy, or both have positive finite entropy, or both have infinite entropy.

\end{enumerate}
\end{remark}

\subsection{Dynamical pseudometrics}

Let $(X,T)$ be a TDS. For $x,z\in X$ we define the \textbf{Besicovitch
	pseudometric} $\rho_{B}$ on $X$ as
\[
\rho_{B}(x,z)=\limsup\limits_{n\rightarrow\infty}\frac{1}{n}\sum_{j=0}%
^{n-1}d(T^{j}x,T^{j}z).
\]
Using standard techniques (for example, see \cite{auslander1959, kwietniak2017} or \cite[Lemma 3.1]{lituye})
one can see that $\rho_{B}$ is uniformly equivalent to another pseudometric on $X$ given by
\[
\rho_{B}^{\prime}(x,z)=\inf\{\delta>0:\overline{D}\{n\geq0:d(T^{n}x,T^{n}%
z)\geq\delta\}<\delta\},
\]
where $\overline{D}(S)$ stands for the \textbf{upper asymptotic density} of $S\subset \N_0$, that is,
\[
\overline{D}(S):=\limsup_{n\to\infty}\frac{\left|S\cap\{0,1,\ldots,n-1\}\right|}{n}.\]
Actually, in this paper we will only use the Besicovitch pseudometric to compare our results using the Feldman-Katok pseudometric with some similar ones based on $\rho_B$.

We will also use \textbf{lower asymptotic density defined} as
\[
\underline{D}(S):=\liminf_{n\to\infty}\frac{\left|S\cap\{0,1,\ldots,n-1\}\right|}{n}.\]

Now we will define the Feldman-Katok pseudometric following \cite{kwietniaklacka17}. For $x,z\in X$, $\delta>0$
and $n\in\N$, we define an $(n,\delta)$\textbf{-match} of $ x$ and $z$ to be an order preserving
	bijection $\pi\colon\Dom(\pi)\rightarrow\Ran(\pi)$ such that $\Dom(\pi
	),\Ran(\pi)\subset\{0,1,\ldots,n-1\}$ and for every $i\in\Dom(\pi)$ we have
	$d(T^{i}x,T^{\pi(i)}z)<\delta$.
	The \textbf{fit} $|\pi|$ of an $(n,\delta)$-match $\pi$ is the cardinality of
	$\Dom(\pi)$.
	 We set
	\[
	\fbar_{n,\delta}(x,z)=1-\frac{\max\{|\pi|:\text{ $\pi$ is an $(n,\delta
			)$-match of $x$ with $z$}\}}{n}.
	\]
	
\begin{definition}
	The \emph{$\fbar_{\delta}$}\textbf{-pseudodistance }between $x$ and $z$ is
	given by
	\[
	\fbar_{\delta}(x,z)=\limsup_{n\rightarrow\infty}\fbar_{n,\delta}(x,z).
	\]
\end{definition}

\begin{definition}
	The \textbf{Feldman-Katok pseudometric of }$(X,T)$\textbf{ } is given by
	\[
	\rho_{FK}(x,z)=\inf\{\delta>0:\fbar_{\delta}(x,z)<\delta\}.
	\]
\end{definition}

\begin{lemma}
	[Fact 22, Lemma 25, and Fact 17 \cite{kwietniaklacka17}] \label{lem:fk-prop} Let $(X,T)$ be a TDS. We have that
	\begin{enumerate}
		\item $\rho_{FK}$ is indeed a pseudometric,
		\item $\rho_{FK}(x,z)\leq\rho_{B}^{\prime}(x,z)$, hence the topology introduced by $\rho_{FK}$ is weaker than topology given by $\rho_{B}$, and
		\item $\rho_{FK}$ is invariant along the orbits, that is, for every $n\ge 0$ and $x,y\in X$ we have that $\rho_{FK}(T^nx,y)=\rho_{FK}(x,y)$. \label{fk-prop-iii}
\end{enumerate}	
\end{lemma}
We will write $B^{FK}_\eps(y)$ for the closed $\rho_{FK}$ ball of radius $\eps$ around $y\in X$.
\begin{remark}
The last property stated in Lemma \ref{lem:fk-prop} seems to be the main source of differences between $\rho_{FK}$ and $\rho_{B}$ as the property is false for the latter pseudometric. To see that let $X=\{0,1\}^\Z$, $T:X\rightarrow X$ be the shift map and $x=(01)^\infty$. One can check that
\[
 \rho_B(x,\sigma(x))>0.
\]
\end{remark}

\section{FK-continuity and sensitivity}\label{sec:FK-cont}
Using the Feldman-Katok pseudometric, we define topological and measure theoretic Feldman-Katok continuity and sensitivity and we provide some basic observations about these properties. We will also see how they relate to mean equicontinuity and mean sensitivity.
\subsection{Equicontinuity and FK-continuity}

Recall that a TDS $(X,T)$ is \textbf{equicontinuous} if $\{T^n: n\ge 0\}$ is an equicontinuous family.

Equicontinuity is a very strong form of stability. For example every equicontinuous subshift (or more generally an expansive TDS) is periodic. The typical example of an equicontinuous system is a rotation on a compact abelian group. Note that a TDS $(X,T)$ is equicontinuous if and only if the map
\[
(X,d)\ni x\mapsto x\in (X,\rho^T)
\]
is continuous where $\rho^T$ is a metric on $X$ given for $x,y\in X$ by
\[
\rho^T(x,y)=\sup_{n\ge 0}d(T^nx,T^ny).
\]
Replacing $\rho^T$ with another dynamically generating (pseudo)metric we obtain new variants of equicontinuity.
We say $x\in X$ is a \textbf{mean equicontinuity point} for $(X,T)$ if for every
$\varepsilon>0$, there exists $\delta>0$ such that if $d(x,y)\leq\delta$ then
$\rho_{B}(x,y)\leq\varepsilon$. 
There are many minimal \textbf{mean equicontinuous} TDSs (every point is a mean equicontinuity point) that are not equicontinuous, but all of them are uniquely ergodic and have zero topological entropy. Nonetheless, within the class of mean equicontinuous systems we can distinguish a whole hierarchy of different (not conjugate) dynamical behaviours \cite{garciajagerye}.

A TDS is \textbf{almost mean equicontinuous} if the mean equicontinuity points form a \textbf{residual} set (countable intersection of dense open sets). There exists almost mean equicontinuous TDSs with positive entropy (and other properties); see \cite[Corollary 4.8]{lituye}, \cite[Theorem 1.4]{garcia2017dynamical}, \cite{li2018chaotic}, and \cite[Corollary 5.7]{garciajagerye}.

Inspired by the previous definition, we introduce the following.
\begin{definition}
Let $(X,T)$ be a TDS. We say that $x\in X$ is an \textbf{Feldman-Katok continuity point (FK-continuity point)} for $(X,T)$ if for every
	$\varepsilon>0$ there exists $\delta>0$ such that if $d(x,y)\leq\delta$ then
	$\rho_{FK}(x,y)\leq\varepsilon$. We denote the set of FK-continuity points by $\EqFK$.
\end{definition}

\begin{definition}	
	We call a TDS $(X,T)$ \textbf{FK-continuous} if every $x\in X$ is an FK-continuity point. We say that $(X,T)$ is \textbf{almost FK-continuous}
	if the set of FK-continuity points is residual.
\end{definition}

	Using Lemma \ref{lem:fk-prop}, it is immediate that every mean equicontinuity point is an FK-continuity point; hence, every
	(almost) mean equicontinuous system is (almost) FK-continuous. We will see later that there exist FK-continuous topological dynamical systems that are not almost mean equicontinuous (Corollary \ref{cor:FK not mean}), and that FK-continuous systems always have zero topological entropy (Corollary \ref{zeroent}).
	
Any almost mean equicontinuous TDS with positive topological entropy is almost FK-continuous but not FK-continuous. 

An even weaker notion of equicontinuity was introduced in \cite{zheng2020new} in order to characterise uniquely ergodic systems.

 Now, we define 	\[
\EqFKe=\left\{  x\in X:\exists\delta>0\text{ }\forall y,z\in
B_{\delta}(x),\text{ }\rho_{FK}(y,z)\leq\varepsilon\right\}  .
\]
\begin{remark}\label{rem:EqFK}
	Note that
	\[
	\EqFK=\bigcap_{n=1}^\infty\EqFKn.
	\]
\end{remark}
The following lemma is proved analogously as for equicontinuous and mean equicontinuous TDSs \cite{akin1996transitive,weakforms}. For completeness, we provide a proof in the Appendix.
\begin{lemma}
	\label{invariant}Let $(X,T)$ be a TDS and $\varepsilon>0$. The sets 
	$\EqFK$ and $\EqFKe$	
	are inversely invariant, that is,
	$T^{-1}(\EqFK)\subseteq \EqFK$ and $T^{-1}(\EqFKe)\subseteq \EqFKe$.
	Furthermore, $\EqFKe$ is open, hence $\EqFK$ is residual.
\end{lemma}
The next two results use invariance of $\rho_{FK}$ along the orbits stated in Lemma \ref{lem:fk-prop} and they are false if we replace $\rho_{FK}$ by the Besicovitch pseudometric.
\begin{lemma}
	\label{invariant2}Let $(X,T)$ be a TDS and $x\in X$.
	If $y\in \EqFK\cap \overline{\orb(x)}$ then $x\in \EqFK$ and $\rho_{FK}(x,y)=0$. 
\end{lemma}
\begin{proof}
	Let $\eps>0$ and $y\in \EqFK\subset \EqFKe$. There is $\delta>0$ such that $B_\delta(y)\subset\EqFKe$ and $\rho_{FK}(y,z)\le \eps$ for every $z\in B_\delta(y)$. Since $y$ is in the closure of the orbit of $x$, there is $n\ge 0$ such that
	$T^nx\in B_\delta(y)$. This implies that $x\in\EqFKe$ by the backward invariance of $\EqFKe$ (Lemma \ref{invariant}) and $\rho_{FK}(x,y)\le \eps$. Since $\eps>0$ was arbitrary,  we conclude that $x\in\EqFK$ and $\rho_{FK}(x,y)=0$.
\end{proof}

\begin{proposition}
	\label{prop1}
	A transitive TDS $(X,T)$ is FK-continuous if and only if for every $x,y\in X$ we have $\rho_{FK}(x,y)=0$.
\end{proposition}

\begin{proof} If $\rho_{FK}(x,y)=0$ for every $x,y\in X$, then trivially $(X,T)$ is FK-continuous. Assume that $(X,T)$ is an FK-continuous transitive TDS (so $\EqFK=X$). Let $x_0\in X$ be a point with a dense orbit and take $y\in X=\overline{\orb(x_0)}$. By Lemma \ref{invariant2}, we obtain that $\rho_{FK}(x_0,y)=0$. Using the triangle inequality, we see that $\rho_{FK}(x,y)=0$ for every $x,y\in X$.
\end{proof}

\begin{remark}\label{exmp:two-point} For the previous proposition, we cannot omit the transitivity assumption, nor replace $\rho_{FK}$ by the Besicovitch pseudometric:
	The TDS $(X,T)$ where $X=\{x_1,x_2\}$ is a discrete metric space and $T=\textrm{Id}_X$ is the identity map, is FK-continuous but $\rho_{FK}(x_1,x_2)>0$.
	Any Sturmian subshift is an example of a transitive and mean equicontinuous TDS  for which $\rho_{B}(x,y)>0$ for some $x,y\in X$.
\end{remark}

\subsection{Sensitivity and Auslander-Yorke dichotomy}\label{sec:AY}
In classical topological dynamics theory, a notion opposite (but not complementary in general) to equicontinuity is sensitivity.
We say that a TDS $(X,T)$ is \textbf{sensitive} if there exists $\eps>0$ such that for every non-empty open set $U$ there are
$x,y\in U$ satisfying $\rho_{B}(x,y)>\varepsilon$. There is also a Besicovitch version of sensitivity; we say that a TDS $(X,T)$ is \textbf{mean sensitive} if there
exists $\varepsilon>0$ such that for every non-empty open set $U$ there exists
$x,y\in U$ satisfying $\rho_{B}(x,y)>\varepsilon$. In the same spirit we introduce sensitivity with respect to the Feldman-Katok pseudometric abbreviated as ``FK-sensitivity''.
\begin{definition}
We say that a TDS $(X,T)$ is \textbf{FK-sensitive} if there is $\varepsilon>0$ such that for every non-empty open set $U$, there exists
	$x,y\in U$ satisfying $\rho_{FK}(x,y)>\varepsilon$.
\end{definition}
\begin{remark}
Note that every FK-sensitive TDS is mean sensitive. Furthermore, a TDS is FK-sensitive if and only if there exists $\eps>0$ such that $\EqFKe$ is empty.
\end{remark}
The well-known Auslander-Yorke dichotomy \cite{auslander1980} states that for minimal TDSs sensitivity complements sensitivity. For transitive systems we have a weaker dichotomy: a transitive TDS is either sensitive or almost equicontinuous \cite{akin1996transitive}. Similar dichotomies hold for mean sensitivity and (almost) mean equicontinuity.
Therefore, it comes as no surprise that an analogous result is true for FK-sensitivity and (almost) FK-equicontinuity. The proof is similar to the Besicovitch pseudometric case and therefore we relegate it to the appendix. For other Auslander-Yorke type dichotomies see \cite{ye2018sensitivity,huang2018analogues,garciajagerye,yurelativization}.

\begin{theorem}[Auslander-Yorke dichotomy for $\rho_{FK}$]
	\label{thm:duality} A transitive TDS is either almost FK-continuous or FK sensitive. A minimal system is either FK-continuous or FK-sensitive.
\end{theorem}

The next result extends \cite[Theorem 8]{auslander1959}, which in turn follows from the results of \cite{oxtoby}. In effect we get a new proof of the analogous statement for $\rho_B$.



\begin{proposition}
	Let $(X,T)$ be a TDS. If there exist at least two $T$-invariant ergodic Borel probability measures with full support then $(X,T)$ is FK-sensitive.
\end{proposition}

\begin{proof}
Assume that $(X,T)$ is not FK sensitive. Let $\mu$ and
	$\nu$ be two ergodic Borel probability measures with full support. Assume that $x_0 \in X$ is a generic point for $\mu$, and $y_0$ is a generic point for $\nu$. Since $x_0$ and $y_0$ are generic points of fully supported measures, their orbits must be dense in $X$. In particular, $(X,T)$ is transitive, so by Theorem \ref{thm:duality} $(X,T)$ it must also be almost FK-continuous. 
By Lemma \ref{invariant2} we get that  $x_0$ and $y_0$ are both FK-continuity points and $\rho_{FK}(x_0,y_0)=0$. Using \cite[Fact 32]{kwietniaklacka17} we conclude that $\mu=\nu$.
\end{proof}

\begin{corollary}
	Let $(X,T)$ be a TDS. If there exist at least two $T$-invariant ergodic Borel probability measures with full support then $(X,T)$ is
	mean sensitive.
\end{corollary}

\subsection{Measure theoretic FK-continuity}
Measure theoretic forms of equicontinuity and sensitivity were introduced in \cite{gilman1987classes} and \cite{huang2011measure} (actually under reasonable conditions these two definitions are equivalent \cite{garcia2017characterization}). The mean/Besicovitch forms of measure theoretic equicontinuity and sensitivity were first studied in \cite{weakforms}. In this section, given a TDS $(X,T)$ and a $T$-invariant Borel probability measure $\mu$, we define  $\mu$-FK-continuity and $\mu$-FK-sensitivity, and we prove their basic properties. Similar to the topological case, one can use Lemma \ref{lem:fk-prop} to see that $\mu$-mean equicontinuity implies $\mu$-FK-continuity, and $\mu$-FK-sensitivity implies $\mu$-mean sensitivity.

\begin{definition}
	\label{mumeanequi} Let $(X,T)$ be a TDS and $\mu$ be a Borel probability measure. We say that $(X,T)$ is \textbf{$\mu$-FK-continuous} if
	for every $\tau>0$ there exists a compact set $M_\tau\subset X$, with $\mu
	(M_\tau)\geq1-\tau$, such that for every $\varepsilon>0$, there exists $\delta>0$
	so that for every $x,y\in M_\tau$, if $d(x,y)\leq\delta$ then
	\[
	\rho_{FK}(x,y)\leq\varepsilon.
	\]
\end{definition}
Clearly, if $(X,T)$ is FK-continuous then it must be $\mu$-FK-continuous for every Borel probability measure.

Given a TDS, we denote the set of $T$-invariant ergodic Borel probability measures by $M^e_1(X,T)$.

\begin{proposition}
	\label{prop2}
	Let $(X,T)$ be a TDS and $\mu\in M^e_1(X,T)$. We have that $(X,T)$ is $\mu$-FK-continuous if and only if there exists a Borel set $H$ with full measure such that $\rho_{FK}(x,y)=0$ for every $x,y\in H$.
\end{proposition}

\begin{proof}Assume that $(X,T)$ is $\mu$-FK-continuous. Fix $0<\tau<1$.
	There exists a Borel set $M_{\tau}\subset X$ (with $\mu(M_{\tau})\geq1-\tau$) and
	$\delta(\eps)>0$ such that if $x,y\in M_{\tau}$ satisfy $d(x,y)\leq\delta(\eps)$, then
	$\rho_{FK}(x,y)\leq\varepsilon$. Let $y_0\in M_\tau$ be such that for every $\delta>0$, we have $\mu(M_\tau\cap B_\delta(y_0))>0$. Since
\[
B_{\delta(\eps)}(y_0)\cap M_\tau \subset B^{FK}_\eps(y_0),
\]
it follows that for every $\eps>0$ we have $\mu(B^{FK}_\eps(y_0))>0$. Set
\[
H=\bigcap_{i=1}^{\infty}\bigcup_{n=0}^\infty T^{-n}(B^{FK}_{1/i}(y_0)).
\]
Since $\mu$ is ergodic, we have $\mu(H)=1$. Reasoning as in the proof of Lemma \ref{invariant2}, it is easy to see that $\rho_{FK}(x,y_0)=0$ for every $x\in H$. By the triangle inequality, we get that $\rho_{FK}(x,y)=0$ for every $x,y\in H$.
The other direction is trivial.
\end{proof}

\begin{remark}
If we do not assume the measure is ergodic then the statement of Proposition \ref{prop2} does not hold. Let $X=\{x,y\}$, $T$ be the identity map, and $\mu=1/2(\delta_x+\delta_y)$, where $\delta_z$ stands for a Dirac measure concentrated on $z$. We have that $(X,T)$ is $\mu$-FK-continuous but $\rho_{FK}(x,y)>0$.
\end{remark}

\begin{definition}
Let $(X,T)$ be a TDS and $\mu$ a Borel probability measure. We say $(X,T)$ is
\begin{itemize}
  \item $\mu$\textbf{-FK-sensitive} if there exists
	$\varepsilon_0>0$ such that for every $A\in\mathcal{B}_{X}^{+}$ there exist
	$x,y\in A$ satisfying $\rho_{FK}(x,y)>\varepsilon_0$,
  \item $\mu$\textbf{-FK-expansive} if there exists
	$\varepsilon'_0>0$ such that
\[\mu\times\mu\left(\left\{  (x,y)\in X\times X:\rho_{FK}(x,y)>\varepsilon'_0\right\}\right)  =1.\]
\end{itemize}
We call $\eps_0$ (respectively, $\eps'_0$) a $\mu$\textbf{-FK-sensitivity constant} (respectively a $\mu$\textbf{-FK-expansivity constant}) of $(X,T)$.
\end{definition}

It is easy to see that the function $\rho_{FK}(\cdot,\cdot)$ is a Borel function on $X\times X$. It follows that
for every $\eps>0$, the set $B_{\varepsilon}^{FK}(x):=\left\{  y\in X:\rho_{FK}(x,y)\leq\varepsilon\right\}$ is also a Borel set.

The proof of the following theorem follows the same lines as the proof of an analogous result presented in \cite[Theorem 26]{weakforms} for the Besicovitch case, which is also similar to a result in \cite{huang2011measure}. For reader's convenience, we present the proof in the appendix.
\begin{theorem}
	\label{strongsensitive}Let $(X,T)$ be a TDS and $\mu\in M^e_1(X,T)$. The following conditions
	are equivalent:
	\begin{enumerate}
	  \item\label{ss:i} $(X,T)$ is $\mu$-FK-sensitive,
	  \item\label{ss:ii} $(X,T)$ is $\mu$-FK-expansive,
	  \item\label{ss:iii} there exists $\varepsilon>0$ such that
	$\mu(B_{\varepsilon}^{FK}(x))=0$ for $\mu$-a.e. $x\in X$, and
	  \item\label{ss:iv} $(X,T)$ is not $\mu$-FK-continuous.
	\end{enumerate}
	\end{theorem}

\section{Loosely Kronecker systems}\label{sec:LK}

In this section we will characterise when a MPS arising from  a TDS $(X,T)$ and a measure $\mu\in M^e_1(X,T)$ is loosely Kronecker using $\rho_{FK}$.

Let $X$ be a compact metrisable space. We write $\bP^k(X)$ for the set of all ordered partitions of $X$ into at most $k$ Borel measurable sets called \textbf{atoms}. For $\mathcal P\in \mathbf P^k(X)$, we write $\mathcal{P}=\{P_0,\ldots,P_{k-1}\}$ regardless of the actual number of nonempty elements in $\mathcal P$; for this, we set $P_j=\emptyset$ for $|\mathcal{P}|\le j<k$.

Assume that $(X,T)$ is a TDS and $\mu$ is a $T$-invariant Borel measure. Let $(X, \mu, T)$ be a MPS and $\mathcal P\in\bP^k(X)$. We identify $\mathcal P$ with a function $\mathcal P\colon X\to \{0,\ldots, k-1\}$ defined by $\mathcal P(x)=j$ for $x\in P_j$.

We endow $\mathbf P^k(X)$  with the pseudometric $\rho_1^\mu$ given for  $\mathcal{P},\mathcal{Q}\in\bP^k(X)$ by
\begin{align*}
\rho_{1}^{\mu}(\mathcal{P},\mathcal{Q})  & =\frac{1}{2}\sum_{j=0}^{k-1}%
\mu(P_{j}\div Q_{j})\\
& =\frac{1}{2}\sum_{j=0}^{k-1}\int_{X}|\chi_{P_{j}}-\chi_{Q_{j}}|\,d\mu\\
& =\mu(\{x\in X:\mathcal{P}(x)\neq\mathcal{Q}(x)\}).
\end{align*}

Note that the definition of $\rho^\mu_1$ takes into account the order of the atoms. 
It is well known that $\rho^\mu_1$ is a complete pseudometric for $\bP^k(X)$, which becomes a metric when we identify partitions $\mathcal{P},\mathcal{Q}\in\bP^k(X)$ satisfying $d^\mu_1(\mathcal{P},\mathcal{Q})=0$.

We define the $n$-th join of $\mathcal P$ as
\begin{align*}
\mathcal P^n & :=\bigvee_{j=0}^{n-1}T^{-j}(\mathcal P)\\ & =\{P_{i_0}\cap T^{-1}(P_{i_1})\cap\ldots\cap T^{-n+1}(P_{i_{n-1}}):P_{i_j}\in\mathcal P\text{ for }0\le j<n\}.
\end{align*}

 \begin{remark}
Note that every atom in $\mathcal P^n$ can also be interpreted as a point (word) in the product space $\mathcal P\times\ldots \times \mathcal P$ ($n$ times).
We define $\mathcal P^\infty$ naturally, and we say $\mathcal P^\infty (x)\in (\mathcal P)^\infty$ is the $ \mathcal P$-itinerary of $x$.
\end{remark}

\begin{definition}
	\label{editdef}
 Let $A$ be a finite set and $u,w\in A^n$. We define the \textbf{edit distance}
\[
\fbar_n(u,w)=1-\frac{k}{n},
\]
where $k$ is the largest among those integers $\ell$ such that for some $0\le i_1<i_2<\ldots<i_\ell<n$ and $0\le j_1<j_2<\ldots<j_\ell<n$ we have
$u(i_s)=w(j_s)$ for $s=1,\ldots,\ell$.
\end{definition}



\begin{theorem}
	[Katok's Criterion \cite{katok77}]
\label{katokcrit}
 A measure-preserving dynamical system $(X,\Sigma,\mu,T)$ is loosely
	Kronecker if and only if for every finite partition $\mathcal{P}$, we have that
	for every $\varepsilon>0$, there exists $N=N(\varepsilon)$ such that for every
	$n\geq N$, there exists a word $w\in\mathcal{P}^{n}$ such that
\[\mu
	(\{w^{\prime}\in\mathcal{P}^{n}:\fbar_{n}(w,w^{\prime})<\varepsilon
	\})\geq1-\varepsilon.\]
\end{theorem}


\begin{lemma}
\label{lemma1}
Let $M\subset X$ be a Borel set such that for every $x,y\in M$ we have $\rho_{FK}(x,y)=0$. Then
for every $\tau>0$, we can find $N>0$ and a set $M_0\subset M$ with $\mu(M\setminus M_0)<\tau$ such that for every $x,y\in M_0$ and $n\ge N$, we have
$\fbar_{n,\tau}(x,y)\le\tau$.
\end{lemma}
\begin{proof}
Let $\tau>0$. If $\mu(M)=0$, then there is nothing to prove. Hence, we assume $\mu(M)>0$ and fix $x_0\in M$. Since $\rho_{FK}(x,y)=0$ for every $x,y\in M$, we have $\fbar_{\tau/3}(x_0,y)<\tau/3$ implying that $\limsup_{n\rightarrow\infty}\fbar_{n,\tau/3}(x_0,y)<\tau/2$. Using essentially the same argument as in the proof of Egorov's Theorem (for example, see \cite[page 33]{steinbook2005}), we obtain there exist $N>0$ and a set $M_0\subset M$, with $\mu(M\setminus M_0)<\tau$, such that for every $y\in M_0$ and $n\ge N$, we have
$\fbar_{n,\tau/3}(x_0,y)<\tau/2$. In particular, for every $n\ge N$ and $y\in M_0$, there exists an $(n,\tau/3)$-match $\pi^y_n\colon \mathcal{D}_n^y\to\mathcal{R}^y_n$ between $x_0$ and $y$ such that $|\mathcal{D}_n^y|=|\mathcal{R}^y_n|> n(1-\tau/2)$.
In other words, the sets $\mathcal{D}_n^y,\mathcal{R}^y_n\subset\{0,1,\ldots,n-1\}$ are such that
\[
d(T^i(x_0),T^{\pi^y_n(i)}(y))<\tau/3\quad\text{for every }i\in \mathcal{D}_n^y.
\]
Let $y,y'\in M_0$, $n\ge N$, and $J=\mathcal{D}_n^y\cap\mathcal{D}^{y'}_n$. We have $|J|>n(1-\tau)$. Clearly $\pi_n=(\pi^{y'}_n)\circ (\pi^{y}_n)^{-1}$ is an $(n,\tau)$-match between $y$ and $y'$ with $\mathcal{D}(\pi_n)=(\pi^{y}_n)^{-1}(J)$ and $\mathcal{R}(\pi_n)=\pi^{y'}_n(J)$, so $|\pi_n|> n(1-\tau)$. This implies the desired result.
\end{proof}

\begin{theorem}
\label{LKcharac} Let $(X,T)$ a TDS and $\mu\in M^e_1(X,T)$. Then, $(X,T)$ is $\mu$-FK-continuous if and only if $(X,\mu,T)$ is loosely Kronecker.
\end{theorem}

\begin{proof}
Assume $(X,\mu,T)$ is loosely Kronecker and $(X,T)$ is $\mu$-FK-sensitive with sensitivity
constant $\varepsilon$. Since $(X,\mu,T)$ is loosely Kronecker, there exist an
isometric TDS $(Y,S)$ (a group rotation equipped with the Haar measure $\nu$) and a Borel set $B\subset Y$,
with $\nu(B)\geq1-\varepsilon/5,$ such that $(X,\mu,T)$ is isomorphic to $(B,\nu_{B},S_{B})$ via $\phi\colon X\to B$. 
By neglecting sets of measure zero we may assume that $\phi$ is a bijection. 
By Lusin's Theorem, there exists a compact
set $M\subset X$ with $\mu(M)\nu(B)\geq 1-\varepsilon/4$ such that 
$\phi|_M\colon M\to\phi(M)$ is uniformly continuous. Hence $\phi^{-1}\colon\phi(M)\to M$ is also uniformly continuous. Note that $\nu_B(\phi(M))=\mu(M)$, thus $\nu(\phi(M))=\nu(B)\mu(M)$.
Given $\eps/2$, we use uniform continuity of $\phi^{-1}$ to find an appropriate $\delta'$, and then we use uniform continuity of $\phi$ to pick $\delta$ for $\delta'$. Using that $S$ is an isometry, it follows that such a
$\delta$ satisfies the following: if $x,y\in M$ and
$d(x,y)\leq\delta$, then
\[d(\phi^{-1}\circ S^{n}\circ \phi(x),\phi^{-1}\circ S^{n}\circ\phi
(y))\leq\varepsilon/2\]
for every $n\in\mathbb{N}$ with $S^{n}\circ\phi(x),S^{n}\circ\phi(y)\in \phi (M)$. 

Let $Y_{M}\subset \phi(M)$ be the set of points in $\phi(M)$ which are $\phi(M)$-regular with respect to the map $S$ and the measure $\nu$, that is, for every $\xi\in Y_M$, we have
\[
\nu(\phi(M))=\mu(M)\nu(B)=\lim_{n\to\infty}\frac{|\{0\le j<n: S^j(\xi)\in \phi(M)\}|}{n}.
\]
For later use, we also define for every $\xi\in Y_M$ the function $c_\xi$, which counts the number of visits of $S$-orbit of $\xi$ to $B$ up to a certain time $n\ge 0$, that is, \[
 c_{\xi}(n):= \left\vert \{0\le i\leq n : S^{i}(\xi)\in B\}\right\vert.\]
Note that for every $\xi\in Y_M$  and $n\ge 0$ such that $S^n(\xi)\in B$, there is $p\in M$ such that  $\phi(p)=\xi$, and
\[
S^n(\xi)=S_B^{c_\xi(n)-1}(\xi)=\phi(T^{c_\xi(n)-1}(p)).
\]

Now, by the pointwise ergodic theorem, we have $\nu(Y_M)=\nu(\phi(M))$. Thus, there is $z\in M$ such that
\[
\mu(B_{\delta/2}(z)\cap M\cap \phi^{-1} (Y_{M}))>0.
\]
Let $x,y\in B_{\delta/2}(z)\cap M\cap \phi^{-1} (Y_{M})$. We will show that
$\rho_{FK}(x,y)\leq\varepsilon/2$; contradicting $\mu$-FK-sensitivity. This ends the first part of the proof since by Theorem \ref{strongsensitive} a system which is not $\mu$-FK-sensitive has to be $\mu$-FK-continuous.

It remains to prove $\rho_{FK}(x,y)\leq\varepsilon/2$. Since $x,y\in \phi^{-1}(Y_M)$ we get that $\phi(x),\phi(y)$ must be $\phi(M)$ regular. We claim that the set $A_M(x,y)$ consisting of those times $n$ that
$S^{n}\circ\phi(x)$ and $S^{n}\circ\phi(y)$ are both in $\phi(M)$ has lower density bounded below by $1-\eps/2$. Indeed, since
\[
 A_{M}(x,y):=\left\{  n\in\mathbb{N}:S^{n}(\phi(x)),S^{n}(\phi(y))\in
\phi(M)\right\},
\]
is the intersection of two sets with asymptotic density
$\mu(M)\nu(B)\geq 1-\varepsilon/4$, then
\[
\underline{D}(A_{M}(x,y))\ge 1-\eps/2.
\]

For $n\in A_M(x,y)$, we set $d(n):=c_x(n)-1$ and $r(n):=c_y(n)-1$.
Note that we have $0\in A_M(x,y)$ and $d(0)=r(0)=0$.
Write $A_M(x,y)$ as an increasing sequence
that is, let
\[
A_M(x,y)=\{a(n) : n\in \mathbb{N}\},\quad\text{where } a(0)<a(1)<a(2)<\ldots<a(n)<\ldots.
\]
We will use the following well-known fact,
\[{\underline {D}}(A_{M}(x,y))=\liminf_{n \rightarrow \infty} \frac{n}{a(n)}.
\]
Now, we define
\[
\mathcal{D}:=\{d(n): n\in \mathbb{N}\},\quad\text{and}\quad \mathcal{R}:=\{r(n): n\in \mathbb{N}\}.
\]
For every $n\in\mathbb{N}$, we have $d(n),r(n)\le a(n)$, so
\[
{\underline {D}}(\mathcal{D})=\liminf_{n \rightarrow \infty} \frac{n}{d(n)}\ge \liminf_{n \rightarrow \infty} \frac{n}{a(n)}={\underline {D}}(A_{M}(x,y)).
\]
Similarly, we get that ${\underline {D}}(\mathcal{R})\ge {\underline {D}}(A_{M}(x,y))$.
For $m\ge 1$, we define $k(m)$ to be the number of ``joint'' visits of $\phi(x)$ and $\phi(y)$ to $\phi(M)$ among the $m$ first visits of each of point to $B$ (under $S$), that is, we set
\[k(m):=
\left\vert\{\ell\ge 0:\ell\in A_M(x,y)\text{ and }\max\{c_x(\ell),c_y(\ell)\}< m\}\right\vert.\]

Let $\mathcal{K}=\{k(1),k(2),\ldots\}$. Clearly
\[
k(m)=\min\{|\mathcal{D}\cap\{0,1,\ldots,m-1\}|,|\mathcal{R}\cap\{0,1,\ldots,m-1\}|\},
\]
so
\[
\liminf_{m\to\infty}k(m)/m\ge\min\{\underline{D}(\mathcal{R}),\underline{D}(\mathcal{D})\}\ge \underline{D}(A_M(x,y))\ge 1-\eps/2.
\]
Now, we define
\begin{align*}
	D_{m} &  :=\{d(0),d(1),\ldots,d(k(m)-1)\}=\mathcal{D}\cap\{0,1,\ldots,m-1\},	 \\
	R_{m} &  :=\{r(0),r(1),\ldots,r(k(m)-1)\}=\mathcal{R}\cap\{0,1,\ldots,m-1\}.
\end{align*}

We set $\pi_m$ as the order preserving bijective map $\pi_m\colon D_{m}
\rightarrow R_{m}$. It is easy to see that $\pi_m$ is an $(m,\eps/2)$-match between
$x$ and $y$. Using this and the fact that $k(m)=\left\vert D_{m}\right\vert =\left\vert R_{m}\right\vert
$, we obtain
\[
\fbar_{m,\eps/2}(x,y)\le 1-\frac{k(m)}{m}.
\]
From a previous estimate, it follows that
\[
\fbar_{\eps/2}(x,y)=\limsup_{m\to\infty}\fbar_{m,\eps/2}(x,y)\le 1-\liminf_{m\to\infty}\frac{k(m)}{m}\le\eps/2.
\]
Thus, $\rho_{FK}(y,z)\leq\varepsilon/2$.

Now we will prove the other direction. By Proposition \ref{prop2}, there exists a Borel set $M\subset X$ with $\mu(M)=1$ such that, $\rho_{FK}(x,y)=0$ for every $x,y\in M$.
Let $\mathcal{P}=\{P_0,P_1,\ldots,P_{k-1}\}\in\bP^k(X)$.
We will prove that $\mathcal{P}$ satisfies the condition of Theorem \ref{katokcrit}.
Let $\eps>0$. Using the regularity of $\mu$, one can see that there exists $\mathcal{R}=\{R_0,R_1,\ldots,R_k\}\in \bP^{k+1}(X)$ with $d^\mu_1(\mathcal{P},\mathcal{R})<\varepsilon/10$, $\mu(R_k)\leq \varepsilon/10$,  and $R_j\subseteq P_j$ for $0\le j<k$ (to find such $\mathcal{R}$ we choose for each $P_j\in \mathcal{P}$  a compact set $R_j\subseteq P_j$ with $\mu (P_j\setminus R_j)$ sufficiently small and then we set $R_k:=X\setminus (R_0\cup R_1\cup\ldots\cup R_{k-1})$).

Furthermore, if $\eps'<\varepsilon$ is sufficiently small, then $R_k$ and $\mathcal{R}$ can be chosen so that if $d(x,y)<\eps'$ for $x,y\in X\setminus R_k$, then $\mathcal{R}(x)=\mathcal{R}(y)$ (we call this condition the \emph{separation property} of $\mathcal{R}$).
Using ergodicity and Egorov's theorem, we can find $N_0\in\N$ and  a Borel set $M_0\subset X$ with $\mu(M_0)>1-\eps/4$ such that for every $z\in M_0$ and $n\ge N_0$, we have
\begin{equation}\label{cond:rk}
\left|\{0\le j <n :T^jz\in R_k\}\right|\le\eps/9.
\end{equation}
By Lemma~\ref{lemma1}, there exists $N_1\in \N$ with $N_1\ge N_0$ and a Borel set $M_1\subset M_0$, with $\mu(M_1)>1-\varepsilon/2$, such that for every $x,y\in M_1$ and $n\ge N_1$, we can find $\ell\ge (1-\varepsilon/9)n$ and two finite sequences of  integers
$0\le i_1<i_2<\ldots<i_\ell<n$, $0\le j_1<j_2<\ldots<j_\ell<n$ such that
\[
d(T^{i_s}(x),T^{j_s}(y))<\eps', \quad\text{for } s=1,\ldots,\ell.
\]
By \eqref{cond:rk}, the set $$S':=\{s\in\{1,\ldots,\ell\}: T^{i_s}(x),T^{j_s}(y)\in X\setminus R_k \}$$ has cardinality at least
$(1-\eps/3)n$. By the separation property of $\mathcal{R}$, for $s\in S'$ we have $\mathcal{R}(T^{i_s}(x))=\mathcal{R}(T^{j_s}(y))$. It follows that for every $n\ge N_1$, if we take any $x\in M_1$ and set $w=\bigcap_{i=0}^{n-1}\mathcal{R}(T^{i}x)=\mathcal{R}^n(x)$, then for every $y\in M_1$ and $w'=\bigcap_{i=0}^{n-1}\mathcal{R}(T^{i}y)$ we have $\fbar_{n}(w,w^{\prime})<\varepsilon/3$. Therefore,
\begin{equation}
\label{eq1}
\forall x\in M_1\quad
\mu
(\{w^{\prime}\in\mathcal{R}^{n}:\fbar_{n}(\mathcal{R}^{n}(x),w^{\prime})<\varepsilon/3
\})\geq \mu(M_1)\ge 1-\varepsilon/2.
\end{equation}

Using \cite[Lemma 46]{kwietniaklacka17}, we can find a set $M_2\subset M_1$ and an integer $N_2\ge N_1$ such that $\mu(M_2)>1-\eps$ and if $z\in M_2$ and $n\ge N_2$, then
\begin{equation}\label{eqn:dpr}
\fbar_n(\mathcal{R}(z),\mathcal{P}(z))<d^\mu_1(\mathcal{P},\mathcal{R})+\varepsilon/10<\varepsilon/3.
\end{equation}

Fix $x_0\in M_2$ and $n\ge N_2$. Set $w_0=\mathcal{R}^n(x_0)$.
Then, by \eqref{eq1}, for every $z\in M_2$ (recall that $M_2\subset M_1$) we have $\fbar(w_0,\mathcal{R}^n(z))<\eps/3$.
Using \eqref{eqn:dpr}, we obtain that for every $z\in M_2$, we have
\[
\fbar_n(\mathcal{P}^n(z),\mathcal{P}^n(x_0))\le \fbar_n(\mathcal{P}^n(z),\mathcal{R}^n(z)) + \fbar_n(\mathcal{R}^n(z),w_0)+\fbar_n(w_0,\mathcal{P}^n(x_0))<\eps.
\]
Therefore,
\[
\mu
(\{w^{\prime}\in\mathcal{P}^{n}:\fbar_{n}(\mathcal{P}^n(x_0)),w^{\prime})<\varepsilon
\})\geq \mu(M_2)\ge 1-\varepsilon.
\]
We conclude the proof using Katok's criterion (Theorem \ref{katokcrit}).
\end{proof}
Summarising all our characterisations we obtain the following corollary.
\begin{corollary}
\label{cor:todos}Let $(X,T)$ be a TDS and $\mu\in M^e_1(X,T)$. The following conditions
are equivalent:
\begin{enumerate}
\item $(X,\mu,T)$ is loosely Kronecker, \label{cond:todos1}
\item there exists a Borel set $M\subset X$ such that $\mu(M)=1$ and $\rho_{FK}(x,y)=0$ for every $x,y\in M$,
\label{cond:todos2}
\item $(X,T)$ is $\mu$-FK-continuous,
\item $(X,T)$ is not $\mu$-FK-sensitive, \label{cond:todos3}
\item $(X,T)$ is not $\mu$-FK-expansive, and \label{cond:todos4}
\item $\mu(B_{\varepsilon}^{FK}(x))>0$ for every  $\varepsilon>0$ and $\mu$-a.e. $x\in X$. \label{cond:todos5}
\end{enumerate}
\end{corollary}

Since loosely Kronecker systems have zero entropy, we obtain the following corollary.
\begin{corollary}
	\label{zeroent}
	Every FK-continuous TDS has zero topological entropy.
\end{corollary}

One can compare Theorem \ref{LKcharac} with the following result (note that its proof in \cite{weakforms} uses different techniques than our proof of Theorem \ref{LKcharac}).
\begin{theorem}[\cite{weakforms}, Corollary 39]
	\label{chakro}
	 Let $(X,T)$ a TDS and $\mu\in M^e_1(X,T)$. Then $(X,\mu,T)$  is a Kronecker system if and only if $(X,T)$ is $\mu$-mean equicontinuous.
\end{theorem}
Actually this result also holds for non-necessarily ergodic Borel invariant probability measures, which is one of the main results of \cite{huang2018bounded}. Another result similar to our Theorem \ref{LKcharac} and \cite[Corollary 39]{weakforms} appeared recently in \cite{cai2019measure}. It characterises rigid MPS.

\section{Topological loosely Kronecker systems}\label{sec:TLK}
Conditions \eqref{cond:todos1} and \eqref{cond:todos2} of Corollary \ref{cor:todos} motivate the following definition.

\begin{definition}
	\label{tlk}
	We say a TDS $(X,T)$ is \textbf{topologically loosely Kronecker} if for every $x,y\in X$ we have $\rho_{FK}(x,y)=0$.
\end{definition}
Recall that a transitive TDS is topologically loosely Kronecker if and only if it is FK-continuous (Proposition \ref{prop1}). The aim of this section is to extend this characterisation and prove that being topologically loosely Kronecker is the same as being a topological model for a loosely Kronecker MPS.

\begin{remark}
Note that  for a TDS $(X,T)$ being ``Besicovitch trivial'', that is, satisfying $\rho_{B}(x,y)=0$ for every $x,y\in X$ is a very strong condition. By \cite[Theorem 1.2]{garcia2017mean}, it holds if, and only if, $(X,T)$ is uniquely ergodic and $(X,\mu,T)$ is trivial (i.e. isomorphic to a one point system) .
\end{remark}

First, we note that the following result is a consequence of \cite[Corollary 34]{kwietniaklacka17}.
\begin{proposition}
	\label{unierg}
Every topologically loosely Kronecker TDS is uniquely ergodic.
\end{proposition}

\begin{remark}
A topologically loosely Kronecker TDS need not to be strictly ergodic. A trivial example is a TDS $(X,T)$, where $X=\{x,y\}$ is a discrete space and $T(x)=T(y)=x$. It is easy to construct a more elaborated examples of this kind.
\end{remark}
The support of a loosely Kronecker measure is not in general uniquely ergodic.
\begin{remark}
 A construction of a minimal non-uniquely ergodic TDS with ergodic measures $\mu$ and $\nu$ such that $\mu$ is loosely Kronecker and $\nu$ is not is given in \cite{complexityloosely}.
\end{remark}
\begin{remark}
Topological representations of the \emph{Pascal adic} (Borel) transformation and other \emph{adic} transformations, were constructed in \cite[Theorem 4.1]{mela2005dynamical} and \cite{mela2006}.
These TDSs (representations) are ``almost minimal'' (all but countably many orbits are dense) and every ergodic invariant measure for an adic transformation is isomorphic to an invariant measure of the topological representation. Furthermore, these measures are loosely Kronecker \cite{janvresse04,mela2006} and fully supported on the TDS (because of the ``almost minimality''). It is known that every adic system has uncountably many ergodic invariant measures, hence these TDSs (representing adic transformations) are not uniquely ergodic and have many fully supported loosely Kronecker invariant measures.
\end{remark}
\begin{theorem}
	\label{uniquelyerg}
Let $(X,T)$ be a TDS. The following conditions are equivalent:
\begin{enumerate}
	\item $(X,T)$ is topologically loosely Kronecker,
	\item $(X,T)$ is uniquely ergodic, with $M^e_1(X,T)=\{\mu\}$, and $(X,\mu,T)$ is loosely Kronecker.
\end{enumerate}
\end{theorem}

\begin{proof}
A topological loosely Kronecker system  is uniquely ergodic by Theorem \ref{uniquelyerg}, and the unique invariant measure must be loosely Kronecker by Theorem \ref{LKcharac}.

For the other direction, assume that $(X,T)$ is uniquely ergodic and $\mu$ is the unique Borel $T$-invariant probability measure. Suppose also that the MPS $(X,\mu,T)$ is loosely Kronecker.
Fix $\eps>0$ and $y,z\in X$. Let $\mathcal{Q}$ be a finite measurable partition of $X$ satisfying $\diam
Q<\varepsilon$ for every $Q\in\mathcal{Q}$. We can assume that $|\mathcal{Q}|\ge 2$.
For each $Q\in\mathcal{Q}$, we
choose a compact set $K_{Q}\subset Q$ such that $\mu(Q\setminus K_{Q}%
)<\varepsilon/(12\cdot|\mathcal{Q}|)$. Since for disjoint $Q,Q^{\prime}%
\in\mathcal{Q}$ the sets $K_{Q}$ and $K_{Q^{\prime}}$ are also disjoint, we have
\[
\gamma=\min\{d(K_{Q},K_{Q^{\prime}}):Q,Q^{\prime}\in\mathcal{Q},\,Q\cap
Q^{\prime}=\emptyset\}>0.
\]
Furthermore, for every $0<\alpha<\min\{\gamma/3,\varepsilon/2\}$
the $\alpha$-hulls $K_{Q}^{\alpha}$ and $K_{Q^{\prime}}^{\alpha}$ of the sets
$K_{Q}$ and $K_{Q^{\prime}}$ are disjoint provided that $Q,Q^{\prime}%
\in\mathcal{Q}$ are disjoint. Furthermore, we choose $0<\delta<\min\{\gamma/3,\varepsilon/2\}$
such that for every $Q\in\mathcal{Q}$ the boundary of $K_{Q}^{\delta}$ is a
$\mu$-null set.
Set
\[
U=\bigcup_{Q\in\mathcal{Q}}K_{Q}^{\delta}.
\]
Let $P_{0}=X\setminus U$ and $\mathcal{P}=\{K_{Q}^{\delta}:Q\in
\mathcal{Q}\}\cup\{P_{0}\}$.
Then, $\mathcal{P}$ is a finite Borel partition
with $\mu(P_0)<\eps/12$ and for any $P\in\mathcal{P}\setminus\{P_{0}\}$ we have $\diam P<\varepsilon$.
Furthermore, for every $n\ge 1$ and every set $P$ in $\mathcal{P}^n$ the boundary of $P$ is a $\mu$-null set, thus every point in $X$ is $P$-regular with respect to $T$ and $\mu$. That is, for every $x\in X$, $n\ge 1$, and $P$ in $\mathcal{P}^n$ we have
\[
\lim_{m\to\infty}\frac{1}{m}\left|\left\{0\le j <m : T^j(x)\in P\right\}\right|=\mu(P).
\]

Applying Katok's criterion (Theorem \ref{katokcrit}), we can find $N$ such that for every $n\ge N$ there is a set $\mathcal{W}_n$ of atoms of $\mathcal{P}^n$ such that $\mu(\mathcal{W}_n)>1-\eps/6$ and for every $w,w'\in\mathcal{W}_n$ we have
$\fbar(w,w')< \eps/3$. Using Egorov's theorem, we see that there is $N'$ such that for every $n\ge N'$, the set $\mathcal{W}_n'$ of atoms $w$ of $\mathcal{P}^n$ such that $P_0$ appears in $w$ at most $n\eps/12$ times satisfies $\mu(\mathcal{W}_n')>1-\eps/6$. Let $n=\max\{N,N'\}$. Note that $\mu(\mathcal{W}_n\cap\mathcal{W}'_n)> 1-\eps/3$.
Since $y,z$ are generic points there is $K$ such that for every $k\ge K$, blocks form $\mathcal{W}_n\cap\mathcal{W}'_n$ appear with the right frequencies in the initial segment of the $\mathcal{P}$-itinerary of $y$ and $z$, respectively. In other words, we have
\begin{gather}\label{appearances}
\left|\left\{0\le j< k: T^j(y)\in\mathcal{W}_n\cap\mathcal{W}'_n\right\}\right|> k(1-\eps/3),\\
\left|\left\{0\le j< k: T^j(z)\in\mathcal{W}_n\cap\mathcal{W}'_n\right\}\right|> k(1-\eps/3).
\end{gather}
In \eqref{appearances}, we listed all appearances of blocks from $\mathcal{W}_n\cap\mathcal{W}'_n$, but we still need to enumerate non-overlapping appearances of these blocks.  Since blocks in $\mathcal{W}_n\cap\mathcal{W}'_n$ have length $n$, the number of non-overlapping appearances is bounded below by $k(1-\eps/3)/n$. Therefore, for a large $k\ge K$ we can find integers $s,t$ and $i(1),\ldots,i(s)$ and $j(1),\ldots,j(t)$ satisfying $i(s)+n\le k$, $j(t)+n
\le k$, $s,t\ge k(1-\eps/3)/n$ and $i(\ell+1)\ge i(\ell)+n$ for $\ell=1,\ldots,s-1$, $j(\ell+1)\ge j(\ell)+n$ for $\ell=1,\ldots,t-1$ and such that
setting $p=\min\{s,t\}$ we have \[
T^{i(\ell)}(y),T^{j(\ell)}(z)\in \mathcal{W}_n\cap\mathcal{W}'_n \quad\text{for }\ell=1,\ldots,p.
\]
Fix $1\le \ell\le p$ and consider $\mathcal{P}^n$ names of $T^{i(\ell)}(y)$ and $T^{j(\ell)}(z)$, denoted, respectively as, $u^{(\ell)}=u^{(\ell)}_1\ldots u^{(\ell)}_{n-1}$ and $v^{(\ell)}=v^{(\ell)}_1\ldots v^{(\ell)}_{n-1}$. Since $u^{(\ell)}, v^{(\ell)}\in \mathcal{W}_n\cap\mathcal{W}'_n$, we can find at most $n\eps/12$ appearances of $P_0$ in each of them and they must satisfy $\bar f(u^{(\ell)}, v^{(\ell)})<\eps/3$. It means that after removing at most $n(\eps/2)$ entries from each, $u^{(\ell)}$ and  $v^{(\ell)}$, we get two identical sequences, which we denote with $\bar u^{(\ell)}$ and $\bar v^{(\ell)}$, and $P_0$ appears neither in $\bar u^{(\ell)}$ nor in $\bar v^{(\ell)}$. We write $\bar u^{(\ell)}$ and $\bar v^{(\ell)}$ explicitly as
\[
\bar u^{(\ell)}=u^{(\ell)}_{a^{(\ell)}(1)}\ldots u^{(\ell)}_{a^{(\ell)}(\bar n)}\quad\text{and}\quad \bar v^{(\ell)}=v^{(\ell)}_{b^{(\ell)}(1)}\ldots v^{(\ell)}_{b^{(\ell)}(\bar n)},
\]
where
\[\{a^{(\ell)}(1),\ldots,a^{(\ell)}(\bar n)\},\,\{b^{(\ell)}(1),\ldots,b^{(\ell)}(\bar n)\}\subseteq\{0,1,\ldots,n-1\}\]
and $\bar n> (1-\eps/2)n$. Since each entry of both, $\bar u^{(\ell)}$ and $\bar v^{(\ell)}$ is different than $P_0$ and they are pairwise equal, we conclude that
\[
\mathcal{P}(T^{i(\ell)+a(\iota)}(y))=\mathcal{P}(T^{j(\ell)+b(\iota)}(z))\quad\text{for every }1\le \ell\le p,\,1\le \iota\le \bar n.
\]
Since $\diam P<\eps$ for $P\in\mathcal{P}$ with $P\neq P_0$, we obtain that
\[
d(T^{i(\ell)+a(\iota)}(y),T^{j(\ell)+b(\iota)}(z))<\eps\quad\text{for every }1\le \ell\le p,\,1\le \iota\le \bar n.
\]
Therefore, the association
\[
i(\ell)+a(\iota)\mapsto b(\ell)+b(\iota)\text{ for every }1\le \ell\le p,\,1\le \iota\le \bar n,
\]
gives us a $(k,\eps)$-match $\pi$ with
\begin{gather*}
\mathcal{D}(\pi)=\{i(\ell)+a(\iota):1\le \ell\le p,\,1\le \iota\le \bar n\}\quad\text{and}\quad  \\
\mathcal{R}(\pi)=\{j(\ell)+b(\iota):1\le \ell\le p,\,1\le \iota\le \bar n\}.
\end{gather*}
We easily have $|\mathcal{D}(\pi)|=|\mathcal{R}(\pi)|\ge (k(1-\eps/3)/n)\cdot (n (1-\eps/2))>(1-\eps)k$. It follows that $\rho_{\text{FK}}(y,z)\le \eps$.
Since $\eps$ is arbitrary, we get $\rho_{\text{FK}}(y,z)=0$ as needed.
\end{proof}

Let $(X,\Sigma_X,\mu,T)$ be a MPS and $(Y,S)$ be a TDS. Recall that we say $(Y,S)$ is a \textbf{topological model} for $(X,\Sigma_X,\mu,T)$ if $(Y,S)$ is uniquely ergodic (with the unique invariant measure denoted by $\nu$) and the MPSs $(X,\Sigma_X,\mu,T)$ and $(Y,\bar{\mathcal{B}}_Y,\nu,S)$ are isomorphic. 

\begin{corollary}
	\label{cor}
	A TDS is a topological model for a loosely Kronecker MPS if and only if it is topologically loosely Kronecker.
\end{corollary}
This result is somewhat surprising, since there is no known purely topological characterisation of topological models of Kronecker systems.
 We will now briefly discuss how the isomorphic invariant properties of MPSs translate to topological models.

For several classes of MPSs there exists \emph{a} topological model (or at least a non-uniquely ergodic topological realisation) with a corresponding topological property \cite{halmos1942operator,lindenstrauss99,donoso2015uniformly,gutman2019strictly}; nonetheless, some topological models may exhibit topological properties rather distant from the properties of the initial MPS.
\begin{theorem}[\cite{lehrer1987topological}]
	\label{lehrer}
	Any ergodic aperiodic measure-preserving dynamical system admits a topologically mixing topological model. In particular, this implies there exists a strictly ergodic topologically mixing TDS such that the associated measure-preserving dynamical system is Kronecker.
\end{theorem}

Minimal topologically mixing TDS are always mean sensitive \cite[Theorem 1.5]{garcia2017dynamical}, so there exist topological models of Kronecker systems that are not mean equicontinuous. Actually, mean equicontinuity characterises a proper subclass of uniquely ergodic Kronecker systems: A TDS $(X,T)$ is mean equicontinuous if and only if it is an isomorphic extension of group rotation i.e. the (continuous) maximal equicontinuous factor map provides a measure theoretic isomorphism between the unique invariant measure for $T$ and the unique invariant measure for its maximal equicontinuous factor \cite{lituye,downarowiczglasner,fuhrmann2018structure}.

A TDS is a topological model of a Kronecker system if and only it is uniquely ergodic and $\mu$-mean equicontinuous \cite{weakforms}; but $\mu$-mean equicontinuity is not a purely topological notion as it refers to a $T$-invariant measure $\mu$ which must be known a priori, while the notion of a topological loosely Kronecker system needs no measure.

A consequence of the previous discussion is the following.
\begin{corollary}
	\label{cor:FK not mean}
	There exists minimal FK-continuous TDS that are not almost mean equicontinuous.
\end{corollary}
\begin{proof}
An example is provided by any minimal topologically mixing topological model of a Kronecker system.
\end{proof}
The next result can be proved directly, but it also easily follows from our previous results.
\begin{corollary}
	\label{cor: TLP invariant}
	Topological loosely Kronecker property is an invariant of topological conjugacy.
\end{corollary}
\begin{proof}
It is well known that unique ergodicity and the loosely Kronecker property are invariant of topological conjugacy. To finish the proof we apply Theorem \ref{uniquelyerg}.
\end{proof}

\section{Examples and applications}\label{sec:examples}
\subsection{Examples of topologically loosely Kronecker systems} \label{subsec:examples}

The following MPSs are loosely Kronecker: finite rank MPSs \cite{ornsteinrudolphweiss82}, horocycle flows \cite{ratner78}, adic transformations with their invariant measures described in \cite{janvresse04,mela2006}, (measurable) distal MPSs (this is a consequence of the fact that the class of loosely
Kronecker MPSs is closed under factors, inverse limits and compact
extensions (\cite{bronstein1979, katok77,ornsteinrudolphweiss82}), and other examples from \cite{feldman76,ratner1979,kanigowski2018,hoffman1999loosely}.

Hence, by Theorem~\ref{uniquelyerg}, every topological model of any of the systems above must be topologically loosely Kronecker. Some of these models are well known TDSs, they include: Morse subshift, Chacon subshift, uniquely ergodic symbolic codings of interval exchange transformations (note that in some sense most of the symbolic codings of the interval exchange transformations are uniquely ergodic \cite{masur1982interval,veech1982gauss}), and minimal nilsystems.



We can also create more examples of topological loosely Kronecker systems using induced flows and regular extensions.

\begin{proposition}
	(\cite[Remark 3.2]{petersen73})
	\label{petersen}
	Let $(X,T)$ be a strictly ergodic TDS and $A\subset X$ be a non-empty clopen set (closed and open). The return map on $A$ is a strictly ergodic TDS.
\end{proposition}

\begin{corollary}
	Let $(X,T)$ be a minimal topologically loosely Kronecker TDS and $A\subset X$ be a non-empty clopen set. The return map on $A$ is a minimal topologically loosely Kronecker TDS.
\end{corollary}
\begin{proof}
	The first return map of a loosely Kronecker system is loosely Kronecker. Simply apply Proposition \ref{petersen} and Theorem \ref{uniquelyerg}.
\end{proof}

For example, if $(X,T)$ is a Sturmian subshift and $A$ a non-trivial cylinder set then the return map will be a topologically loosely Kronecker (and not necessarily mean equicontinuous) TDS.

Another class of examples studied in \cite{petersen73} are the \emph{primitives} of TDS with respect to a function (definied in \cite[Section 2]{petersen73}). By arguments given on that paper (Remark 3.2 and Section 3), it is easy to see that the primitives of topologically loosely Kronecker TDS are also strictly ergodic and Kakutani equivalent to an topologically loosely Kronecker TDS; hence, also topologically loosely Kronecker. An example of one of these systems is the Sturmian subshift with "doubled" 1s.
\begin{definition}
We say $(X,T)$ is a \textbf{regular extension} of $(Y,\nu,S)$, if there
exists a continuous factor $\phi:(X,T)\rightarrow(Y,S)$ such that $\nu(\{y\in Y:|\phi^{-1}(y)|=1\})$.
\end{definition}

One way to generate a regular extension of uniquely ergodic TDS is with a symbolic extension (or codings) generated by a measurable partition with zero measure boundaries.

The following lemma is known (for example see the corollary on page 383 in \cite{petersen73}).

\begin{lemma}
	\label{regext}
If $(Y,S)$ is uniquely ergodic TDS and $(X,T)$ is a regular extension of $(Y,S)$ then $(X,T)$ is uniquely ergodic TDS.
Furthermore, if we denote the unique invariant measures of, respectively, $(X,T)$ and $(Y,S)$ by $\mu$ and $\nu$, then the MPSs $(X,\mu,T)$ and $(Y,\nu,S)$ are isomorphic.
\end{lemma}

\begin{corollary}
Every regular extension of a topologically loosely Kronecker system must
also be topologically loosely Kronecker.
\end{corollary}
\begin{proof}
	Regular extensions yield an isomorphism and hence a Kakutani equivalence. Use Lemma \ref{regext} and Theorem \ref{uniquelyerg} to obtain the result.
\end{proof}

Minimal regular extensions of equicontinuous TDSs were characterised using a stronger form of mean equicontinuity in \cite{garciajagerye}.
\subsection{FK-sensitivity corollaries}
In this section we gather some corollaries of Auslander-Yorke type dichotomies proved in section \ref{sec:AY}.


\begin{corollary} There exists a MPS that admits no FK-sensitive topological model.
\end{corollary}
\begin{proof}
	If a MPS is loosely Kronecker, then any topological model for it must be topologically loosely Kronecker (Theorem \ref{uniquelyerg}). Thus it cannot be FK-sensitive.
\end{proof}
Compare the previous result with Theorem \ref{lehrer} (and the discussion after it). This indicates than FK-sensitivity is much stronger than mean sensitivity.

Possibly the most interesting corollary of this subsection is the following.

\begin{corollary}
	\label{cor:min-pos-ent}
	Every minimal TDS with positive topological entropy is FK-sensitive.
\end{corollary}
\begin{proof}
	If a minimal TDS is not FK-sensitive, then it must be FK-continuous by Theorem \ref{thm:duality}, and hence topologically loosely Kronecker  by Proposition \ref{prop1}.  Now, Corollary \ref{zeroent} concludes the proof.
\end{proof}
To our knowledge, FK-sensitivity is the strongest form of sensitivity that is implied by minimality and positive topological entropy.
\begin{corollary}
	Every minimal topological model for a positive entropy ergodic MPS is FK-sensitive.
\end{corollary}
\begin{proof}
	Assume the model is not FK-sensitive. Then, using Theorem \ref{thm:duality}, it must be FK-continuous. A contradiction arises using Corollary \ref{zeroent}.
\end{proof}


\section{Acknowledgements}
The authors would like to thank  Rafael Alcaraz Barrera, Martha {\L}{\k{a}}cka, Karl Petersen, and Ay\c{s}e \c{S}ahin  for interesting conversations. Parts of this work were prepared during visits of one of us to the other one institution. The warm hospitality of JU and UASLP is gratefully acknowledged. Both authors would like to thank Avery, Jasmine and Ronnie Pavlov for creating an opportunity for us to work on the paper while visiting Denver. We would also like to thank Xiangdong Ye for comments on a preliminary version of this paper.
The research of DK was supported by the National Science Centre, Poland, grant no. 2018/29/B/ST1/01340. FGR was supported by the SEP/CONACyT Ciencia Básica project 287764.
\appendix
\section{Appendix}\label{sec:app}
In this section we provide some proofs for completeness.
We remind the reader that the FK-continuity points are denoted by $\EqFK$ and
	\[
\EqFKe=\left\{  x\in X:\exists\delta>0\text{ }\forall y,z\in
	B_{\delta}(x),\text{ }\rho_{FK}(y,z)\leq\varepsilon\right\}  .
	\]

\begin{proof}[Proof of Lemma \ref{invariant}]
	Let $\varepsilon>0$ and $x\in T^{-1}\EqFKe$.  
Since $Tx\in \EqFKe$ we can find $\eta>0$ such that if $y,z\in B_\eta(Tx)$ 
then $\rho^{FK}(y,z)\leq\varepsilon$. 
By continuity of $T$ there also exists $\delta>0$ such that for every $y'\in X$
if	$d(x,y')\leq\delta$ then $d(Tx,Ty')\leq\eta$. Therefore for $y',z'\in B_{\delta}(x)$ we have $Ty',Tz'\in B_{\eta}(Tx)$, so $\rho_{FK}(y',z')=\rho_{FK}(Ty',Tz')\leq\varepsilon$. We conclude that $x\in \EqFKe$, 
and so $\EqFKe$ is inversely invariant. To see that $\EqFKe$ is open, we fix $x\in \EqFKe$
and pick $\delta>0$ from the defining property of $\EqFKe$. Now it is easy to see that if $x'\in B_{\delta/2}(x)$ and $y,z\in B_{\delta/2}(x')$
then $y,z\in B_{\delta}(x)$, so $\rho^{FK}(y,z)<\varepsilon$ and $\EqFKe$ is open.
We finish the proof by invoking Remark \ref{rem:EqFK}.
	\end{proof}

\begin{proof}
	[Proof of Theorem \ref{thm:duality}] Assume that $(X,T)$ is a transitive TDS. It is enough to show that if $(X,T)$ is not FK-sensitive, then $\EqFK$ is nonempty. In fact, we will show that $\EqFK$ must be residual. It will then follow from Lemma \ref{invariant} that every point $x\in X$ with dense orbit is also in $\EqFK$. Since in a minimal TDS every point has dense orbit we will get $X=\EqFK$.
Therefore we assume from now on that $(X,T)$ is not FK-sensitive, so $\EqFKe$ is nonempty for every $\eps>0$. In a transitive system every nonempty, open and backward invariant set must be dense, so we conclude that $\EqFKn$ is open and dense in $X$ for $n=1,2,\ldots$. Invoking Remark \ref{rem:EqFK} and the Baire category theorem we see that $\EqFK$ is residual in $X$.
\end{proof}

The following lemma is simlar to \cite[Lemma 25]{weakforms}, but it gives more information.
\begin{lemma}
	\label{ecu}Let $(X,\mu,T)$ be an ergodic system and $\eps>0$. The function
	\[
	x\mapsto \mu\left(B_\eps^{FK}(x)\right)
	\]
is $\mu$-almost everywhere constant and its value is either $0$ or $1$. Furthermore, it
	is $\mu$-almost
	everywhere equal to $\mu\times\mu\left\{  (x,y)\in X\times X:\rho_{FK}%
	(x,y)\leq\varepsilon\right\}$.
\end{lemma}

\begin{proof} Recall that $B_\eps^{FK}(x)=\left\{y:\rho_{FK}(x,y)\leq\varepsilon\right\}$ and note that $B_\eps^{FK}(x)$ is an invariant set by Lemma \ref{lem:fk-prop}\eqref{fk-prop-iii}. Therefore for each $x\in X$ and $\eps>0$ it holds that $\mu(B_\eps^{FK}(x))$ is either $0$ or $1$.

By the same lemma the mapping $x\mapsto \mu\left(B_\eps^{FK}(x)\right)$ is also $T$-invariant, so ergodicity implies that $\mu\left(B_\eps^{FK}(x)\right)$ is constant for $\mu$-a.e. $x\in X$. Using Fubini's Theorem we obtain that
	\begin{multline*}
	\mu\times\mu\left\{  (x,y)\in X\times X:\rho_{FK}(x,y)\leq\varepsilon\right\}
	=\\=\int_{X}\int_{X}1_{\left\{  (x,y):\rho_{FK}(x,y)\leq\varepsilon
		\right\}  }d\mu(y)d\mu(x)
	 =\int_{X}\mu\left(B^{FK}_\eps(x)\right)\, 
	d\mu(x).
	\end{multline*}
Therefore $\mu\left(B^{FK}_\eps(x)\right)$ equals $\mu\times\mu\left\{  (x,y)\in X\times X:\rho_{FK}(x,y)(x,y)\leq\varepsilon\right\}$ for $\mu$-a.e. $x\in X$ and the latter value is either $0$ or $1$. 	
\end{proof}


\begin{proof}
	[Proof of Theorem \ref{strongsensitive}] [\eqref{ss:i}$\Rightarrow$\eqref{ss:iii}] Suppose $(X,T)$ is a $\mu$-FK-sensitive TDS and $\eps$ is its $\mu$-FK-sensitivity
	constant. If \eqref{ss:iii} does not hold, then there
	exists $x\in X$ such that $\mu(B_{\varepsilon/2}^{FK}(x))>0$. For $y,z\in B_{\varepsilon/2}^{FK}(x)$,
	we have that $\rho_{FK}(y,z)\leq\varepsilon$. This contradicts the assumption that $(X,T)$ is $\mu$-FK sensitive.

\noindent [\eqref{ss:iii}$\Rightarrow$\eqref{ss:ii}]  Let $\eps$ be as in Theorem \ref{strongsensitive}\eqref{ss:iii}. Using Lemma \ref{ecu} we obtain that
\[\mu\times\mu\left(\left\{  (x,y)\in X\times X:\rho_{FK}(x,y)\leq\varepsilon\right\}\right)  =0,
\]
which means that $\eps$ is a $\mu$-FK-expansivity constant for $(X,T)$. 	

\noindent [\eqref{ss:ii}$\Rightarrow$\eqref{ss:i}] Let $\eps_0=\varepsilon'_0$, where $\varepsilon'_0$ is a $\mu$-FK-expansivity constant for $(X,T)$. Fix $A\in\Sigma^{+}$ and note that $A\times A\in(\Sigma\times\Sigma)^{+}$. By $\mu$-FK-expansivity
	we can find $(x,y)\in A\times A$ such that $\rho_{FK}(x,y)>\varepsilon_0$, which shows that $\eps_0$ is a $\mu$-FK-sensitivity constant for $(X,T)$.

\noindent [\eqref{ss:iii}$\Leftrightarrow$\eqref{ss:iv}] It follows immediately from 	Lemma \ref{ecu} and Proposition \ref{prop2}.
\end{proof}
\bibliographystyle{alpha}
\bibliography{refs}
\end{document}